\documentclass[12pt]{amsart}
\usepackage{latexsym}
\usepackage{hyperref}

\newcommand{\ignore}[1]{}

\DeclareMathOperator{\ad}{ad}

\newcommand{\F}{\mathbb F}

\newtheorem{dummy}{Dummy}

\newtheorem{lemma}[dummy]{Lemma}
\newtheorem*{lemma*}{Lemma}
\newtheorem{theorem}[dummy]{Theorem}

\theoremstyle{definition}

\theoremstyle{remark}
\newtheorem{rem}[dummy]{Remark}

\begin{document}

\bibliographystyle{alpha}
\author{Sandro Mattarei}
\title[Constituents and chain lenghts]{Constituents of graded Lie algebras\\ of maximal class\\ and chain lengths of thin {Lie} algebras}

\begin{abstract}
Thin Lie algebras are infinite-dimensional graded Lie algebras $L=\bigoplus_{i=1}^{\infty}$,
with $\dim(L_1)=2$ and satisfying a {\em covering property:}
for each $i$, each nonzero $z\in L_i$ satisfies $[zL_1]=L_{i+1}$.
It follows that each homogeneous components $L_i$ is either one- or two-dimensional,
and in the latter case is called a {\em diamond}.
Hence $L_1$ is a diamond, and if there are no other diamonds then $L$ is a graded Lie algebra of maximal class.

We present simpler proofs of some fundamental facts on graded Lie algebras of maximal class, and on thin Lie algebras, based on a uniform method,
with emphasis on a polynomial interpretation.
Among else, we determine the possible values for the most fundamental parameter of such algebras,
which is the dimension of their largest metabelian quotient.
\end{abstract}
\maketitle

\section{Introduction}\label{sec:intro}

A {\em graded Lie algebra of maximal class} is a graded Lie algebra $L=\bigoplus_{i=1}^{\infty}L_i$
with $\dim L_1=2$, $\dim L_i\le 1$ for $i>1$, and $[L_i,L_1]=L_{i+1}$ for all $i$.
Their name refers to the fact that, in case $1<\dim(L)<\infty$, such Lie algebras are nilpotent
of nilpotency class as large as it can be compared to their dimension, hence precisely one less than their dimension.
Taking the latter as definition (and without an assumption of being graded), Lie algebras of maximal class are also referred to as
{\em alg\'{e}bres filiformes} in the literature.
The general wisdom is that those are too complicated to ever admit a comprehensive classification,
even in characteristic zero.
A reasonable approach in their study is shifting focus to to the graded Lie algebra associated with a suitable filtration,
the most natural being the lower central series.
The resulting graded Lie algebras of maximal class are then those of our definition.
Over a field a characteristic zero, taking the associated graded Lie algebra is too drastic a simplification,
as it is not difficult to see that there then at most two nonisomorphic graded Lie algebras of maximal class of each given finite dimension $\dim(L)$.
However, it was noted by Shalev in~\cite{Sha:max} that the landscape might look more complicated in positive characteristic,
as he constructed countably many insoluble (infinite-dimensional) graded Lie algebras of maximal class.

In this paper we make the simplifying assumption that all Lie algebras considered are infinite-dimensional,
whence graded Lie algebras of maximal class satisfy $\dim(L_i)=1$ for $i>1$.
It will be clear that each of our results actually remains valid in the finite-dimensional case
under an assumption that $\dim(L)$ is sufficiently large.
A systematic investigation of graded Lie algebras of maximal class started in~\cite{CMN}.
In particular, constructions of new such algebras from a given one were described,
which when done repeatedly together with some limit processes produce,
over an arbitrary field $F$ of positive characteristic, $\max\{|F|,\aleph_0\}$ pairwise non-isomorphic graded Lie algebras of maximal class.
A classification of graded Lie algebras of maximal class was achieved in~\cite{CN} for $p$ odd, and in~\cite{Ju:maximal} for $p=2$,
in the sense of proving that any graded Lie algebras of maximal class can be obtained through the procedures described in~\cite{CMN}.

Important quantities associated with a graded Lie algebra of maximal class are the {\em constituent lengths}.
Those numbers can intrinsically be defined in terms of relative codimensions of the Lie powers of $L^2$, namely,
$\ell=\ell_1=\dim\bigl(L^2/(L^2)^2\bigr)+1$
and
$\ell_r=\dim\bigl((L^2)^r/(L^2)^{r+1}\bigr)$ for $r>1$.
In particular, the starting point of investigating graded Lie algebras of maximal class is establishing that $\ell$
must equal either $\infty$ or twice a power of the characteristic, say $\ell=2q$ with $q>1$ a power of $p$.
We will revisit the original proof of that result given in~\cite{CMN},
but before discussing that we introduce another class of graded Lie algebras,
which relate to graded Lie algebras of maximal class in various ways.

A {\em thin} Lie algebra is an infinite-dimensional (as we assume in this paper) graded Lie algebra $L=\bigoplus_{i=1}^{\infty}L_i$
with $\dim L_1=2$ and satisfying the following {\em covering property:}
for each $i$, each nonzero $z\in L_i$ satisfies $[zL_1]=L_{i+1}$.
This implies at once that homogeneous components of a thin Lie algebra are at most two-dimensional.
Those components of dimension two are called {\em diamonds,}
hence $L_1$ is a diamond, and if there are no other diamonds then $L$ is a graded Lie algebra of maximal class.
It is convenient, however, to explicitly exclude graded Lie algebras of maximal class from the definition of thin Lie algebras.
Thus, a thin Lie algebra must have at least one further diamond past $L_1$, and we let $L_k$ be the earliest (the {\em second} diamond).

The term {\em diamond} originates from a lattice-theoretic characterization of thin Lie algebras motivated by~\cite{Br}.
In fact, any graded ideal $I$ of a thin Lie algebra is comprised between two consecutive Lie powers of $L$, in the sense that $L^i\supseteq I\supseteq L^{i+1}$ for some $i$,
and hence the lattice of graded ideals looks like a sequence of diamonds (a name for the lattice of subspaces of a two-dimensional space) connected by {\em chains}.
We will not formally assign numerical lengths to those chains but it should be clear that they are important in describing the structure of a thin Lie algebras.
Knowing those lengths amounts to knowing the degrees in which the diamonds occur.

In particular, the most fundamental invariant of a thin Lie algebra is the degree $k$ of the second diamond.
It turns out (as was proved in~\cite{AviJur}, but see Section~\ref{sec:second_diamond} for further information)
that $k$ can only be one of $3$, $5$, $q$, or $2q-1$, where $q$ is a power of the characteristic when positive.
In particular, only $3$ and $5$ can occur in characteristic zero,
and subject to a restriction such thin Lie algebras (infinite-dimensional as we assume throughout) were shown to belong to precisely three isomorphism types~\cite{CMNS},
associated to $p$-adic Lie groups of types $A_1$ and $A_2$~\cite{Mat:thin-groups}.
In contrast, the values $q$ and $2q-1$ for $k$ occur for two broad classes of thin Lie algebras built from certain nonclassical finite-dimensional
simple modular Lie algebras, and also to thin Lie algebras obtained from graded Lie algebras of maximal class through various constructions.
Overall discussions of those two classes of thin Lie algebras, with references, can be found in~\cite{AviMat:A-Z} and~\cite{CaMa:Hamiltonian}, respectively.

The occurrence of powers $q$ of the characteristic $p$ in these results can ultimately be traced,
not unexpectedly, to the fact that $(X+1)^n=X^n+1$ in $\F_p[x]$ precisely when $n$ is a power of $p$,
or generalizations of that basic fact which we discuss in Section~\ref{sec:preli}.
One of the goals of this paper is to provide a revised and simplified exposition
of the determination of the fundamental parameters $\ell$ and $k$ of graded Lie algebras of maximal class
and thin Lie algebras, emphasizing the polynomial viewpoint.
We deal with the length $\ell$ of the first constituent in Section~\ref{sec:constituents},
and with the degree $k$ of the second diamond in Section~\ref{sec:second_diamond}.
One of the reasons for this approach is removing much of the educated guessing in the choice of
particular Lie product calculations rather than others in the original proofs,
which was developed after reliance on the results of computer calculations.
Our proofs are generally shorter than the original proofs, but more importantly,
they are more natural in the sense that after a minimal setup the proofs themselves produce the correct statement to be proved.
Isolating the key reasons for the admissible values for $\ell$ and $k$ from accessory calculations
makes the arguments more flexible for use in other settings.

In particular, one such setting which has been only partially investigated
is the study of other types of graded Lie algebras of maximal class, which are not generated by $L_1$,
but by an element of degree $1$ and one of degree $n>1$.
For $n=2$ those were classified by Shalev and Zelmanov in~\cite{ShZe:narrow-Witt} in characteristic zero,
and by Caranti and Vaughan-Lee in~\cite{CVL00,CVL03} in positive characteristic (odd and then two).
Partial results for arbitrary $n$ where then found by Ugolini~\cite{Ugo:thesis} (see also~\cite{Ugo:type_n}),
and a classification for $n$ equal to the characteristic $p$ was obtained in~\cite{Sca:thesis}.
An improved exposition of part of the results of~\cite{Sca:thesis} is given in~\cite{IMS}.
In that paper, the determination of the possibilities for the {\em first constituent length}
has greatly benefitted from the {\em polynomial} approach introduced here.

The author is grateful to Marina Avitabile for discussions on an early draft of this document which was privately circulated since 2005.
That preliminary, unrefined and longer version, was later cited in~\cite{AviMat:A-Z,AviMat:diamonds,AJM} under the same title,
but included further, essentially disjoint material, which eventually became~\cite{Mat:sandwich}.
Some of those citations refer to the content of~\cite{Mat:sandwich} rather than this paper.

\section{Preliminaries}\label{sec:preli}

As announced in the Introduction, for clarity of exposition we assume all Lie algebras in this paper to have infinite dimension.
Allowing their dimension to be finite but large enough would suffice, and inspection of our proofs would reveal how large is precisely enough,
but we would not be able to improve on the original dimension bounds in~\cite{CMN} and~\cite{AviJur}, which were sharp in each case.
We allow the characteristic $p$ to be zero, although most of the interest lies in positive characteristic.

We use the left-normed convention
for iterated Lie products, hence $[abc]$ stands for $[[ab]c]$.
We also use the shorthand
$[ab^i]=[ab\cdots b]$,
where $b$ occurs $i$ times.
A fundamental calculation device in most arguments, both in~\cite{CMN,CN,Ju:maximal} and in a number of papers on thin Lie algebras,
is the following {\em generalized Jacobi identity,}
which can be easily proved by induction:
\begin{equation}\label{eq:iterated-Jacobi}
[v[y z^j]]
=\sum_{i=0}^{j}(-1)^i \binom{j}{i}
[vz^iy z^{j-i}].
\end{equation}
In a typical application of Equation~\eqref{eq:iterated-Jacobi}
most of the summands will vanish,
usually because $[vz^iy]=0$ except for a few values of $i$ (at most two in this paper).

Because of the recurrent use of Equation~\eqref{eq:iterated-Jacobi}, in all papers in this research area
frequent use is made of Lucas' theorem, a basic tool for evaluating a binomial coefficient $\binom{a}{b}$ modulo a prime $p$:
if $a$ and $b$ are non-negative integers with $p$-adic expansions $a=a_0+a_1p+\cdots+a_rp^r$ and $b=b_0+b_1p+\cdots+b_rp^r$, then
$
\binom{a}{b}= \prod_{i=0}^{r} \binom{a_i}{b_i} \pmod{p}.
$
Lucas' theorem is easily proved by expanding both sides of
$
(1+X)^a=\prod_{i=0}^{r}(1+X^{p^i})^{a_i}
$
in the polynomial ring $\F_p[X]$.
As an example of application of Lucas' theorem, the simple fact, mentioned in the Introduction, that the condition
$(X+1)^n=X^n+1$ in $\F_p[X]$
entails that $n$ is a power of $p$,
follows by reading the $p$-adic expansion of $n$ from the vanishing modulo $p$ of the selected binomial coefficients
$\binom{n}{p^r}\equiv 0\pmod{p}$ for $p^r<n$.

In this paper we deliberately {\em avoid} using Lucas' theorem for the most part.
As an example, deducing that $n$ is a power of $p$ from the equation
$(X+1)^n=X^n+1$ in $\F_p[X]$
can be done without direct appeal to Lucas' theorem (even though, admittedly, in a less elementary way):
that condition implies that the map $\alpha\mapsto\alpha^n$ is an automorphism of any finite field of characteristic $p$,
whence it must equal some power of the Frobenius automorphism $\alpha\mapsto\alpha^p$.

In our proof that the first constituent length $\ell$ of a graded Lie algebra of maximal class $L$ must equal to twice a power of $p$,
in Theorem~\ref{thm:first_constituent}, we will use the following variation:
if $(1+X)^{2n}\equiv 1\pmod{X^n}$ in $\F_p[X]$, then $n$ equals a power of $p$.
In fact given that $\binom{2n}{k}$ is a multiple of $p$ for $0<k<n$,
a simple application of Lucas' theorem (using $\binom{k}{p^r}\equiv 0\pmod{p}$ for $p^r<n$) yields the conclusion.
However, the following simple manipulations avoid direct appeal to Lucas' theorem.
Because of the symmetry $\binom{2n}{2n-k}=\binom{2n}{k}$, we are actually given that $(X-1)^{2n}=X^{2n}+aX^n+1$ in $\F_p[X]$, for some $a\in\F_p$.
Substituting $1$ for $X$ we get $a=-2$, whence $(X-1)^{2n}=(X^n-1)^2$.
Consequently, $(X-1)^n=X^n-1$, and hence $n$ is a power of $p$.

In Theorem~\ref{thm:first_chain} we will show that the degree $k$ of the second diamond in a thin Lie algebra,
which can easily be seen to be odd, can only be $3$, $5$, $q$, or $2q-1$, for some power $q$ of $p$.
In that case, after some preparations, an application of Equation~\eqref{eq:iterated-Jacobi}
will give us a congruence involving binomial coefficients, which in turn
amounts to the condition
\begin{equation}\label{eq:congr_2n+1}
(X+1)^{2n+1}(1-nX)\equiv 1+(n+1)X\pmod{X^n}
\end{equation}
in the polynomial ring $\F_p[X]$, where $k=2n+1$.
The following result will then give almost the desired conclusion, and in~Remark~\ref{rem:neater} we discuss how the spurious values will be excluded.

\begin{lemma}\label{lemma:first_chain}
Let $p$ be a prime or zero, and suppose the odd integer $2n+1>1$ satisfies
\[
\binom{2n+1}{j+1}\equiv n\binom{2n+1}{j}\pmod{p}
\quad\text{for $0<j<n-1$.}
\]
Then $2n+1$ equals $3$, $5$, $7$, $q$, $2q-1$, or $2q+1$, for some power $q$ of $p$.
\end{lemma}

\begin{proof}
Note that the congruence holding in the stated range is equivalent to Equation~\eqref{eq:congr_2n+1}.
The congruence is trivially satisfied for $j=1$.
No further values of $j$ belong to the given range unless $2n+1>7$, which we assume now.
For $j=2$ and $p\neq 2,3$ the congruence is equivalent to $n(2n+1)(-n-1)\equiv 0\pmod{p}$, whence $n\equiv 0,-1/2,-1\pmod{p}$.
We consider each case in turn, thus covering $p=2$ and $p=3$ as well.

If $p\neq 2$ and $n\equiv -1/2\pmod{p}$,
then $\binom{2n+1}{1}\equiv\binom{2n+1}{2}\equiv 0\pmod{p}$, and inductively we obtain
$\binom{2n+1}{j}\equiv 0\pmod{p}$
for $0<j<n$.
Because of the symmetry $\binom{2n+1}{2n+1-j}=\binom{2n+1}{j}$ this also holds for $n+1<j<2n+1$, and so
$(X+1)^{2n+1}=X^{2n+1}+aX^{n+1}+aX^n+1$ in $\F_p[X]$, for some $a\in\F_p$.
Since this polynomial is a $p$-th power (as the exponent $2n+1$ is a multiple of $p$) its derivative
$a(n+1)X^{n+1}+anX^n$ is the zero polynomial, whence $a=0$.
Therefore, $(X+1)^{2n+1}=X^{2n+1}+1$ in $\F_p[X]$,
and as we have discussed earlier this implies $2n+1=q$, a power of $p$.

For the other two cases we write our condition in the equivalent form
\[
[X^j](X+1)^{2n+1}(nX-1)=0,
\quad\text{for $1<j<n$,}
\]
where the polynomial is viewed in $\F_p[X]$.
Here $[X^j]f(X)$ stands for the coefficient of $X^j$ in the polynomial $f(X)$.

Thus, if $n\equiv -1\pmod{p}$ our condition means
\[
[X^j](X+1)^{2n+2}=0,
\quad\text{for $1<j<n$,}
\]
and clearly also for $j=1$.
Hence
$(X+1)^{2n+2}=X^{2n+2}+aX^{n+2}+bX^{n+1}+aX^n+1$ in $\F_p[X]$, for some $a,b\in\F_p$.
Since this polynomial is a $p$-th power its derivative $aX^{n+1}-aX^{n-1}$ is the zero polynomial, and hence $a=0$.
As we have seen earlier this forces $2n+2$ to be twice a power of $p$, and so $2n+1=2q-1$.

Finally, if $n\equiv 0\pmod{p}$ our condition becomes
\[
[X^j](X+1)^{2n+1}=0,
\quad\text{for $1<j<n$.}
\]
Again by symmetry this also holds for $n+1<j<2n+1$, hence
$(X+1)^{2n+1}=X^{2n+1}+X^{2n}+aX^{n+1}+aX^n+X+1=(X^{2n}+aX^n+1)(X+1)$,
for some $a\in\F_p$,
and so
$(X+1)^{2n}=X^{2n}+aX^n+1$.
Hence $2n$ is twice a power of $p$, and so $2n+1=2q+1$.
\end{proof}

\begin{rem}\label{rem:neater}
A somehow neater version of Lemma~\ref{lemma:first_chain} would assume that the congruence holds in the extended range $0<j<n$.
For $p\neq 2$ that would rule out the cases where $2n+1$ equals $7$ or $2q+1$,
thus producing a sharp conclusion in our application to thin Lie algebras in Section~\ref{sec:second_diamond},
where $k=2n+1$ will be the degree of the second diamond.
Unfortunately, our Lie algebraic calculations of Section~\ref{sec:second_diamond} naturally lead to the hypothesis of Lemma~\ref{lemma:first_chain} as stated,
and those spurious values for $k$ will have to be excluded by other means.
In fact, Equation~\eqref{eq:diamond} and the discussion which follows it implies that $(k-1)/2$ cannot be a multiple of $p$
(or the covering property will be violated).
That will exclude $k=2q+1$,
and $k=7$ will be ruled out in an ad-hoc manner (for $p\neq 2,7$) in the proof of Theorem~\ref{thm:first_chain}.
\end{rem}

In the proof of Theorem~\ref{thm:any_constituent} we will make use of the following congruence, which holds
if $q$ is a power of the prime $p$:
\begin{equation}\label{eq:binomial}
(-1)^{a}\binom{a}{b}\equiv
(-1)^{b}\binom{q-1-b}{q-1-a}
\pmod{p}
\quad\textrm{for $0\le b\le a<q$.}
\end{equation}
When $q=p$ this is an easy application of Wilson's theorem $(p-1)!\equiv -1\pmod{p}$.
A proof of the general case based on different manipulations can be found in~\cite[Section~4]{Mat:binomial}.
We give here a different and perhaps more conceptual proof of Equation~\eqref{eq:binomial}.
Because $(-1)^{a}\binom{a}{0}=(-1)^{a}$,
$(-1)^{a}\binom{a}{a}=(-1)^{a}$, and
$(-1)^{q-1}\binom{q-1}{b}\equiv (-1)^{b}\pmod{p}$,
on the boundary of the triangular region
$0\le b\le a<q$
under consideration
the residue modulo $p$ of the signed binomial coefficient
$(-1)^{a}\binom{a}{b}$
takes alternate values $1$ and $-1$.
In particular, those boundary values are invariant under any of the six symmetries of this triangular region.
But then the value modulo $p$ of $(-1)^{a}\binom{a}{b}$ must likewise be invariant in the interior of the region,
being uniquely determined by the boundary values and the equation
\[
(-1)^{a+1}\binom{a+1}{b+1}+(-1)^{a}\binom{a}{b}+(-1)^{a}\binom{a}{b+1}=0,
\]
which is a slightly more symmetric version of the most basic binomial identity $\binom{n+1}{k+1}=\binom{n}{k}+\binom{n}{k+1}$.
In conclusion, the value modulo $p$ of $(-1)^{a}\binom{a}{b}$ can be written in six different ways
according to the six symmetries of the triangular shape.
Because the group is generated by any two reflections, each of those six expressions can be obtained
by suitable repeated application of $(-1)^{a}\binom{a}{b}=(-1)^{a}\binom{a}{a-b}$ and Equation~\eqref{eq:binomial}.

\section{Constituents of graded Lie algebras of maximal class}\label{sec:constituents}

Because of our blanket infinite-dimensionality assumption, in this paper a graded Lie algebras of maximal class is a graded Lie algebra $L=\bigoplus_{i=1}^{\infty}L_i$
with $\dim L_1=2$, $\dim L_i=1$ for $i>1$, and $[L_i,L_1]=L_{i+1}$ for all $i$.
The study of graded Lie algebras of maximal class in~\cite{CMN,CN,Ju:maximal} relied on
associating an infinite sequence to any such algebra $L$,
the {\em sequence of two-step centralizers} of $L$,
which essentially describes the adjoint action of $L_1$ on a graded basis of $L$.
The elements of the sequence are the centralizers $C_{L_1}(L_i)$, which are one-dimensional subspaces of $L_1$.
Therefore, they can be parametrized by points on a projective line,
and by scalars in the underlying field plus a symbol $\infty$ once a choice of a choice of homogeneous generators for $L$ is made.
There is a natural way to split the sequence into a union of adjacent finite sequences called {\em constituents}.
For the purposes of this paper we will only need to know about their lengths, whose definition we recall below.
As we mentioned in the Introduction, constituent lengths have an intrinsic definition
in terms of relative codimensions in the sequence of Lie powers $(L^2)^r$ of $L^2$.
A proof that this is equivalent to the one we give below
is given in~\cite[Section~5]{IMS} in a more general setting.

Suppose $L$ is a non-metabelian (infinite-dimensional) graded Lie algebra of maximal class.
Hence if $C_{L_1}(L_2)=\F y$ is the first two-step centralizer, then there is a component $L_i$
not centralized by $y$, say $L_{\ell}$ is the earliest.
The parameter $\ell$ is {\em the first constituent length}.
As is customary, we then set $C_{L_1}(L_{\ell})=\F x$.
Thus, we have
\begin{subequations}
\begin{align}
[yx^iy]&=0\quad\text{for $0\le i<\ell-1$},
 \label{eq:first_const}\\
[yx^{\ell}]&=0.
 \label{eq:first_const_end}
\end{align}
\end{subequations}
So far the only information one can obtain on $\ell$ is that it must be even, otherwise
\begin{equation}\label{eq:ell_even}
0=[yx^{(\ell-1)/2}\,[yx^{(\ell-1)/2}]]
=(-1)^{(\ell-1)/2}[yx^{\ell-1}y]
\end{equation}
yields a contradiction.

Proceeding further, the calculation
\begin{equation}\label{eq:2k}
0=[yx^{\ell-1}\,[yx^{\ell-1}]]=
[yx^{\ell-1}yx^{\ell-1}],
\end{equation}
together with our blanket assumption that $L$ is infinite-dimensional,
shows
that there must be another homogeneous component past $L_{\ell}$ which is not centralized by $y$,
say $L_{\ell+\ell_2}$ is the earliest.
We call $\ell_2$ {\em the second constituent length}.
Equation~\eqref{eq:2k} implies that $\ell_2<\ell$.

Analyzing the first and second constituents of $L$ will be sufficient to deduce that $\ell$ must equal twice some power of the characteristic
(which is therefore positive), as we state formally in Theorem~\ref{thm:first_constituent}.
However, because we will consider arbitrary constituents in Theorem~\ref{thm:any_constituent}, we set up the following notation for convenience.
We set
$v_1=[yx^{\ell-1}]$,
and for $r>1$ we let positive integers $\ell_r$ and homogeneous elements
$v_r$ of $L$ be recursively defined by
\begin{subequations}
\begin{align}
&[v_{r-1}yx^iy]=0\quad\text{for $0\le i<\ell_r-1$},
 \label{eq:const}\\
&[v_ry]\neq 0,\quad\text{where }
v_r:=[v_{r-1}yx^{\ell_r-1}].
 \label{eq:const_end}
\end{align}
\end{subequations}
Note that we are not asserting anything about $[v_rx]$ for $r>1$, which may vanish or not (and then be a scalar multiple of $[v_ry]$).
The positive integers $\ell_r$ (where we may view $\ell_1=\ell$), are the {\em constituent lengths} of $L$.
No further constituent length can exceed the first,
because if it were $\ell_r>\ell$ then we would get a contradiction by computing
\[
0=[v_{r-1}[yx^{\ell}]]
=\sum_{i=0}^{\ell}(-1)^i \binom{\ell}{i}
[v_{r-1}x^iy x^{\ell-i}]
=[v_{r-1}yx^{\ell}].
\]
Hence we have $\ell_r\le\ell$ for all $r$, and actually $\ell_2<\ell$ because of Equation~\eqref{eq:2k}.
Note that this shows, recursively, that such positive integers $\ell_r$ actually exist.

We also have $\ell_r>1$, because at least one of each pair of consecutive two-step centralizers equals $y$;
this follows from the following calculation as in~\cite[Lemma~3.3]{CMN}:
if $u$ is a homogeneous element of $L$ such that $[ux]=v_r$, then
\begin{equation}\label{eq:vyy}
0=[u[xyy]]=[uxyy]-2[uyxy]+[uyyx]=[uxyy]=[v_ryy].
\end{equation}
Incidentally, this conclusion for all $r$ implies $[Lyy]=0$, or $(\ad y)^2=0$.
In characteristic not two this can be expressed by saying that $y$ is a {\em sandwich element} of $L$.
Extending this fact, under certain conditions, to when $L$ is a thin Lie algebra, is the subject of~\cite{Mat:sandwich}.

We now show that $\ell/2$ is a lower bound for the length of any constituent.

\begin{lemma}\label{lemma:constituent_bound}
Let $L$ be a non-metabelian graded Lie algebra of maximal class, with constituent lengths $\ell=\ell_1,\ell_2,\ell_3,\ldots$
in the order of occurrence.
Then $\ell_r\ge \ell/2$ for all $r$.
\end{lemma}

\begin{proof}
Equation~\eqref{eq:first_const} and an application of the generalized Jacobi identity yield
$[yx^{i-1}[yx^{j-1}]]=0$
for $i,j>0$ with $i+j\le\ell$.
In particular, we have
\begin{equation}\label{eq:diag}
[yx^j[yx^{j-1}]]=0
\quad\text{for $0<j<\ell/2$}.
\end{equation}
We will use this to prove $\ell_r\ge \ell/2$ by induction on $r$.

Thus, assume $\ell_r\ge \ell/2$ holds for some $r\ge 1$, which is trivially true when $r=1$.
According to Equations~\eqref{eq:const} and~\eqref{eq:const_end}, together with the generalized Jacobi identity, for $0<j<\ell_r$ we have
$[v_{r-1}yx^{\ell_r-j-1}[yx^{j-1}]]=0$ and $[v_{r-1}yx^{\ell_r-j-1}[yx^j]]=(-1)^jv_r$.
In particular, because $\ell_r\ge \ell/2$ these hold for $0<j<\ell/2$, and in this range
together with Equation~\eqref{eq:diag} they imply
\begin{align*}
0
&=
[v_{r-1}yx^{\ell_r-j-1}\,[yx^j\,[yx^{j-1}]]]
\\&=
[v_{r-1}yx^{\ell_r-j-1}\,[yx^j]\,[yx^{j-1}]]
-[v_{r-1}yx^{\ell_r-j-1}\,[yx^{j-1}]\,[yx^j]]
\\&=
(-1)^j[v_ry\,[yx^{j-1}]].
\end{align*}
However, Equations~\eqref{eq:const} and~\eqref{eq:const_end} with $r$ increased by $1$, together with the generalized Jacobi identity, yield
\[
[v_ry\,[yx^{\ell_{r+1}-1}]]=(-1)^{\ell_{r+1}-1}[v_{r+1}y]\neq 0.
\]
We conclude $\ell_{r+1}\ge\ell/2$ as desired.
\end{proof}

Note that the characteristic of the field is not mentioned in the statement of Lemma~\ref{lemma:constituent_bound}, and plays no role in its proof.
Armed with Lemma~\ref{lemma:constituent_bound} we can now give a very concise proof of the following result,
which was originally Theorem~5.5 in~\cite{CMN} (and summarized the contents of of Lemmas 5.3 and 5.4 there).

\begin{theorem}\label{thm:first_constituent}
Let $L$ be a graded Lie algebra of maximal class, over a field of characteristic $p$.
Then the first constituent of $L$
has length $2q$
for some power $q$ of $p$.
Furthermore, if $p$ is odd then the second constituent has length $q$.
\end{theorem}

\begin{proof}
Using the generalized Jacobi identity and Equations~\eqref{eq:first_const} and~\eqref{eq:first_const_end} we find
\begin{equation*}
0=[yx^{j-1}[yx^{\ell}]]=
(-1)^{k-j}\binom{\ell}{j}
[yx^{\ell-1}yx^j].
\end{equation*}
By definition of $\ell_2$ according to Equations~\eqref{eq:const} and~\eqref{eq:const_end}
we have
$[yx^{\ell-1}yx^j]=[v_1yx^j]\neq 0$ for $0<j<\ell_2$, and we infer
$\binom{\ell}{j}\equiv 0\pmod{p}$
in that range.
Because $\ell_2\ge\ell/2$ from Lemma~\ref{lemma:constituent_bound},
we have
$\binom{\ell}{j}\equiv 0\pmod{p}$
for all $0<j<\ell$ except possibly for $j=\ell/2$.
Because $\ell$ is even, we have seen in Section~\ref{sec:preli} how this forces $\ell/2$ to be a power $q$ of $p$, as desired.

As a consequence, $\binom{\ell}{\ell/2}=\binom{2q}{q}=2$, whence $\ell_2\le \ell/2$ if $p$ is odd, and therefore $\ell_2=\ell/2$.
\end{proof}

According to~\cite{Ju:maximal}, in characteristic two
the second constituent of $L$ can take any length which is allowed by the following general result on constituent lengths,
apart from the highest, this restriction being due to Equation~\eqref{eq:2k}.

We now give a revised proof of~\cite[Proposition~5.6]{CMN},
which is the following statement on the possible lengths of arbitrary constituents.
As recalled earlier, beware that constituent lengths as originally defined in~\cite{CMN}
were then all increased by one according to an updated, more natural definition introduced in~\cite{CN}.

\begin{theorem}\label{thm:any_constituent}
Let $L$ be a graded Lie algebra of maximal class with first constituent of length $2q$.
Then the lengths of the constituents can only take the values
\[
2q,\quad\text{and}\quad 2q-p^s,
\quad\text{with $p^s\le q$.}
\]
\end{theorem}

\begin{proof}
We know from Lemma~\ref{lemma:constituent_bound} and an earlier observation that every constituent length $\ell_r$ satisfies $q\le \ell_r\le 2q$.
Now suppose $\ell_r<2q-1$ for some $r$.
If $\ell_r\le j< 2q-1$, then because $j-\ell_r+1<q\le\ell_{r-1}$ there exists a homogeneous element $u$ such that $v_{r-1}=[ux^{j-\ell_r+1}]$.
(In fact, $u=[v_{r-2}yx^{\ell_{r-1}+\ell_r-2-j}]$, taking into account our convention on $v_0$ when $r=2$.)
Noting that $[uy]=0$ and using the generalized Jacobi identity we find
\[
0=[u[yx^jy]]= [u[yx^j]y]
=(-1)^{j-\ell_r+1}\binom{j}{j-\ell_r+1}[v_{r-1}yx^{\ell_r-1}y].
\]
Because $[v_{r-1}yx^{\ell_r-1}y]=[v_ry]\neq 0$ we infer
$\binom{j}{j-\ell_r+1}\equiv 0\pmod{p}$
for $\ell_r\le j<2q-1$.
However, the value modulo $p$ of that binomial coefficient can be manipulated as follows:
\begin{align*}
\binom{j}{j-\ell_r+1}
&\equiv
\binom{j-q}{j-\ell_r+1}=
\binom{j-q}{\ell_r-q-1}
\\&\equiv
(-1)^{j+\ell_r-1}\binom{2q-\ell_r}{2q-j-1}
\pmod{p}.
\end{align*}
Here we have used Equation~\eqref{eq:binomial} for the last step, and
the following special case of Lucas' theorem for the first step:
$\binom{j+q}{i}\equiv\binom{j}{i}\pmod{p}$
if $q$ is a power of $p$ and $i<q$.
This follows at once from the congruence $(X+1)^{j+q}\equiv(X+1)^j(X^q+1)\pmod{p}$.
In conclusion, we have obtained
$\binom{2q-\ell_r}{i}\equiv 0\pmod{p}$
for $0<i<2q-\ell_r$.
As recalled in Section~\ref{sec:preli} this yields the desired conclusion that $2q-\ell_r$ is a power of $p$.
\end{proof}

\section{The degree of the second diamond in a thin Lie algebra}\label{sec:second_diamond}

The possibilities for the degree $k$ of the second diamond $L_k$ of a thin Lie algebra
were determined in~\cite{AviJur}, extending more specialized results in~\cite{CMNS}
and also building on results of~\cite{CaMa:thin} and~\cite{CaJu:quotients}.
It was proved in~\cite{AviJur} that the second diamond of an infinite-dimensional thin Lie algebra $L$
can only occur in an odd degree $k$ of the form $3$, $5$, $q$, or $2q-1$,
where $q$ is a power of $p$. 
(We reserve the letter $k$ for the degree of the second diamond throughout this section.)
That $3$ and $5$ are the only possibilities in characteristic zero was already known from~\cite{CMNS}.

The conclusions on $k$ of~\cite{AviJur}, as we have just stated them, turned out to be only justified for $p\neq 2$ at the time,
because they depended on previous results in~\cite{CaJu:quotients} for $p\neq 2$, and in~\cite{Ju:quotients} for $p=2$,
and the latter paper was later found to contain errors.
Those results allowed the authors of~\cite{AviJur} to assume that the quotient $L/L^k$ is metabelian,
and their arguments remain valid under that assumption.
(Thin Lie algebras were allowed to be finite-dimensional in~\cite{CaJu:quotients,Ju:quotients,AviJur}, thus providing more precise information,
which we disregard here.)
According to~\cite{CaJu:quotients} the quotient $L/L^k$ is necessarily metabelian when $p\neq 2$.
It need not be when $p=2$, and the paper~\cite{Ju:quotients} claimed a complete structural description of $L/L^k$ which is incorrect.
This was rectified in~\cite{AJM}, which turned out to involve a much more complex argument than the original fallacious one of~\cite{Ju:quotients} but,
in particular, confirmed its original consequence that $k$ must have the form $2q-1$ when $p=2$ and $L/L^k$ is not metabelian.

In this section we give a new and substantially shorter proof of the main result of~\cite{AviJur},
or rather of its consequence for an infinite-dimensional thin Lie algebra $L$,
following the same idea behind our results of the previous section.
According to~\cite{CaJu:quotients} for $p\neq 2$ ,
and to~\cite{AJM} for $p=2$, which amends \cite{Ju:quotients},
we may assume that $L/L^k$ is metabelian, which means $(L^2)^2\subseteq L^k$.
Obviously $L_2=[L_1,L_1]$ cannot be a diamond, but examples in~\cite{CMNS} and~\cite{GMY}
show that $L_3$ can be, hence $k=3$ is a possibility.
Thus, from now on we assume $k>3$, and so $L_3$ is not a diamond.

Assuming $[L_2L_1]$ to be one-dimensional allows us to choose a nonzero $y\in L_1$ such that $[L_2,y]=0$.
We also choose $x\in L_1$ not to be a scalar multiple of $y$, and hence $x$ and $y$ generate $L$.
Thus, we have
\begin{subequations}
\begin{align}
&[yx^iy]=0\quad\text{for $0\le i<k-2$},
\label{eq:yx..y}\\
&[vy]\neq 0,\quad\text{where }
v:=[yx^{k-2}].
\label{eq:yx..}
\end{align}
The above conditions are sufficient to determine $k$ and define $v$, in a similar fashion as Equations~\eqref{eq:const} and~\eqref{eq:const_end}.
A calculation already employed in the previous section, namely, Equation~\eqref{eq:ell_even} but with $k-1$ in place of $\ell$,
shows that $k$ must be odd.
Next, we conveniently set $v:=[yx^{k-2}]$.
Hence $[vx]$ and $[vy]$ are linearly independent and span the diamond $L_k$.

The next homogeneous component $L_{k+1}$ is spanned by $[vxx]$, $[vxy]$, $[vyx]$, and $[vyy]$,
but Equation~\eqref{eq:vyy} with $v$ in place of $v_r$ shows
\begin{align}\label{eq:diamond_vyy}
&[vyy]=0.
\end{align}
Because $[vyL_1]=L_{k+1}$ according to the covering property,
$L_{k+1}$ is spanned by $[vyx]$ and, in particular, is one-dimensional.
Also, computing
\begin{align*}
0&=[yx^{(k-1)/2}\,[yx^{(k-1)/2}]]=
\\
&=(-1)^{(k-3)/2}\binom{(k-1)/2}{(k-3)/2}[vyx]+(-1)^{(k-1)/2}[vxy],
\end{align*}
we find
\begin{align}\label{eq:diamond}
&[vxy]=\frac{k-1}{2}[vyx].
\end{align}
Now if we had $[vxy]=0$ then, because $[vxx]$ is a scalar multiple of $[vyx]$, some nonzero linear combination $z$ of $[vx]$ and $[vy]$
would satisfy $[zx]=0$ and $[zy]=0$, thus violating the covering property.
Consequently, $(k-1)/2$ cannot be a multiple of $p$,
meaning $k\not\equiv 1\pmod{p}$ if $p\neq 2$, and $k\equiv -1\pmod{4}$ if $p=2$.

Note that $[vxx]$ may be zero, or not.
Setting $x'=x+\alpha y$ and computing
\[
[vx'x']=[vxx]+\alpha[vxy]+\alpha[vyx]=[vxx]+\frac{k+1}{2}\alpha[vyx]
\]
we see that by replacing $x$ with $x'$ for a suitable choice of $\alpha$ (which does not affect any of our previous conclusions) we may attain
\begin{align}\label{eq:vxx}
&[vxx]=0\quad\text{if $k\not\equiv -1\pmod{p}$},
\end{align}
which is analogous to Equation~\eqref{eq:first_const_end}.
The thin algebras with second diamond $L_{2q-1}$ and all diamonds of finite type, constructed in~\cite{CaMa:thin},
show that $k\equiv -1\pmod{p}$ is a genuine exception here.
We will make only a marginal use of Equation~\eqref{eq:vxx} in this paper, in ruling out $k=7$ at the end of the proof of Theorem~\ref{thm:first_chain}.

Moving past degree $k+1$, we will show that not all homogeneous components past $L_k$ can be centralized by $y$,
hence we may let an integer $h>1$ be defined by
\begin{align}
&[vyx^iy]=0\quad\text{for $0<i<h-1$},
\label{eq:vyx^iy}\\
&[vyx^{h-1}y]\neq 0.
\label{eq:vyx..y}
\end{align}
\end{subequations}
However, until we prove the existence of $h$ in the proof of Theorem~\ref{thm:first_chain},
for uniformity of notation we set $h=\infty$ in case $[vyx^iy]=0$ for all $i>0$
(with the convention that $\infty$ is larger than any integer).
Note the similarity of these equations to Equations~\eqref{eq:const} and~\eqref{eq:const_end}.
Thus, if $h>1$ then $L_{k+h-1}$ is the next homogeneous component after $L_k$ which is not centralized by $y$.
(It will follow from a discussion at the end of this section that $h=1$ can possibly occur only when $k=3$.)

Now we are ready to establish the following result, which determines all possibilities for $k$.

\begin{theorem}\label{thm:first_chain}
Let $L$ be a thin Lie algebra over a field of arbitrary characteristic $p$ (possibly zero), with second diamond $L_k$, and $L/L^k$ metabelian.
Then $k$ equals one of $3$, $5$, $q$ or $2q-1$
for some power $q$ of $p$ in case $p>0$.
However, if $p=2$ then $k$ equals either $3$ or $2q-1$.
\end{theorem}

Our proof of Theorem~\ref{thm:first_chain} is very similar to that of Theorem~\ref{thm:first_constituent},
but includes preparations analogous to Lemma~\ref{lemma:constituent_bound}.

\begin{proof}
We assume $k>3$ as we may, so we can make full use of the equations found above.
First we need a lower bound on $h$ in terms of $k$, which is analogous to Lemma~\ref{lemma:constituent_bound} and is proved by
the same characteristic-free argument.
Thus, as in the proof of Lemma~\ref{lemma:constituent_bound} one finds
$[yx^{i-1}[yx^{j-1}]]=0$
for $i,j>0$ with $i+j\le k-1$.
In particular, we have
\[
[yx^j[yx^{j-1}]]=0
\quad\text{for $0<j<(k-1)/2$}.
\]
It follows that
\begin{align*}
0
&=
[yx^{k-j-2}\,[yx^j\,[yx^{j-1}]]]
\\&=
[yx^{k-j-2}\,[yx^j]\,[yx^{j-1}]]
-[yx^{k-j-2}\,[yx^{j-1}]\,[yx^j]]
\\&=
(-1)^j[vy\,[yx^{j-1}]],
\end{align*}
again for $0<j<(k-1)/2$.
Because
$
[vy\,[yx^{h-1}]]=(-1)^{h-1}[vyx^{h-1}y]\neq 0
$
we conclude $h\ge (k-1)/2$.

The rest of the proof is similar in structure to the proof of Theorem~\ref{thm:first_constituent},
with $[yx^k]$ playing a similar role to $[yx^{\ell}]$ there.
Although the vanishing of $[yx^k]=[vxx]$
is generally not available to us (Equation~\eqref{eq:vxx}), we still have
\[
[yx^k\,[yx^{j-1}]]=
[vxx\,[yx^{j-1}]]=
\sum_{i=0}^{j-1}(-1)^i\binom{j-1}{i}[vx^{2+i}yx^{j-1-i}]=0
\]
for $0<j<h-1$.
We deduce
\begin{align*}
0&=[yx^{j-1}\,[yx^k]]
\\&=
(-1)^{k-j-1}\binom{k}{j+1}
[vyxx^j]
+(-1)^{k-j}\binom{k}{j}
[vxyx^j]
\\&=(-1)^{j}\left(\binom{k}{j+1}
-\frac{k-1}{2}\binom{k}{j}\right)
[vyxx^j]
\end{align*}
for $0<j<h-1$.
Because $[vyxx^j]\neq 0$ in that range due to Equation~\eqref{eq:vyx..y}, we obtain
\begin{equation}\label{eq:binom_k}
\binom{k}{j+1}\equiv \frac{k-1}{2}\binom{k}{j}\pmod{p}
\quad\text{for $0<j<h-1$},
\end{equation}
which is to be read as an equality when $p=0$.
Because $\binom{k}{k-1}=k$, $\binom{k}{k}=1$, and $\binom{k}{k+1}=0$, the congruence in Equation~\eqref{eq:binom_k} cannot hold for both $j=k-1$ and $j=k$,
hence $h$ is finite and $h\le k+1$.

Because  $h\ge(k-1)/2$ the congruence in Equation~\eqref{eq:binom_k} holds, in particular, for $0<j<(k-3)/2$.
Now Lemma~\ref{lemma:first_chain} with $2n+1=k$, together with our restriction that $(k-1)/2$ is not a multiple of $p$
(as a consequence of Equation~\eqref{eq:diamond}),
shows that $k$ can only take one of the claimed values, or possibly $7$.

Of course $k=7$ is a possibility when $p=2$ (having the form $2q-1$) and when $p=7$ (having the form $q$).
To exclude $k=7$ for $p\neq 2,7$ we may argue as follows.
If $k=7$ then we have found above $h\ge 3$, whence $[vyxy]=0$,
but if $p\neq 2,7$ then Equation~\eqref{eq:binom_k} fails for $j=2$, and so $h=3$, whence $[vyxxy]\neq 0$.
Since $[vxy]=3[vyx]$ according to Equation~\eqref{eq:diamond}, and $p\neq 3$ because $(k-1)/2$ cannot be a multiple of $p$,
those two conditions may be equivalently written as $[vxyy]=0$ and $[vxyxy]\neq 0$.
We then find
$0=[vx[xyy]]=[vxxyy]-2[vxyxy]$,
and hence $[vxxyy]\neq 0$.
However, this contradicts the fact that $[vxx]=0$ may be attained after suitably redefining $x$, according to Equation~\eqref{eq:vxx}.
\end{proof}

In a similar fashion as how the proof of Theorem~\ref{thm:first_constituent} determined the length of the second constituent for $p$ odd,
our proof of Theorem~\ref{thm:first_chain} provides partial information on the quantity $h$ defined by Equations~\eqref{eq:vyx^iy} and~\eqref{eq:vyx..y}.
In fact, in the course of the proof of Theorem~\ref{thm:first_chain}
we have shown $(k-1)/2\le h\le k+1$.
The upper bound was used to ensure the finiteness of $h$, but can be refined
by a direct application of Equation~\eqref{eq:binom_k} in each case,
of which we give a sample here, referring to results from other papers for more complete conclusions.

If $k=2q-1$ and $p\neq 2$ then, because $\binom{2q-1}{q-1}=\binom{2q-1}{q}=1$, Equation~\eqref{eq:binom_k} fails for $j=q-1$.
Hence $h\le q$, and together with the inequality $(k-1)/2\le h$ we find that $h$ equals either $q-1$ or $q$.
However, taking $h=q$ contradicts Equation~\eqref{eq:vyx..y} by computing
\begin{multline*}
0=[v[yx^{q-1}y]]
=[v[yx^{q-1}]y]-[vy[yx^{q-1}]]
\\=
[vyx^{q-1}y]-(q-1)[vxyx^{q-2}y]-[vyx^{q-1}y]
=
-[vyx^{q-1}y],
\end{multline*}
where we have expanded the Lie products using Equation~\eqref{eq:vyx^iy} and the fact that $[vxx]$ a scalar multiple of $[vyx]$.
In conclusion, if $k=2q-1$ and $p\neq 2$ we have shown $h=q-1$.
One can also prove that $\dim(L_{k+h})=2$ in this case, and hence $L_{3q-2}$ is {\em the third diamond} of $L$, see~\cite{CaMa:thin}.

If $k=2q-1$ and $p=2$, then Equation~\eqref{eq:binom_k} fails the first time for $j=2q-1$,
and hence our arguments so far only show $q-1\le h\le 2q$.
In fact, the case of characteristic two is genuinely more complicated.
A connection with certain graded Lie algebras of maximal class established in~\cite[Section~5]{CaMa:thin}
produces examples of thin Lie algebras with $k=2q-1$ that, in characteristic two, attain any value $h=2q-2^s-1$ with $2^s\le q$.

If $k=q$, then Equation~\eqref{eq:binom_k} fails for $j=q-1$, and hence $(q-1)/2\le h\le q$.
This is a poor conclusion, as for $p>3$ and $q>5$ it is shown in~\cite{CaMa:Nottingham} that $h$ can only take the values $q-1$ or $q$.
The former value occurs, in particular, when $L_{2q-1}$ is the third diamond of $L$.
However, it is natural to call {\em fake diamonds} certain one-dimensional components of thin Lie algebras,
and here the third diamond can be fake, see~\cite{CaMa:Nottingham}.
When $L_{2q-1}$ is a fake diamond of one particular type, $h$ takes the other value $q$.
Thin Lie algebras with $k=q$ have been named {\em Nottingham algebras} because a special case is related to the {\em Nottingham group},
see~\cite{Car:Nottingham,AviMat:A-Z,AviMat:Nottingham_structure}
for a variety of examples and structural results.
In particular, it is shown in~\cite[Section~3]{AviMat:Nottingham_structure} that, when properly interpreted in case of fake diamonds,
the degree difference between any two consecutive diamonds in a Nottingham algebra equals $q-1$.

If $k=5\neq p$, then Equation~\eqref{eq:binom_k} fails for $j=2$, and hence $2\le h\le 3$.
However, $h=3$ can be ruled out by a similar application of~\eqref{eq:yx..y} as we did above to exclude $h=q$ in case $k=2q-1$,
see also~\cite{CMNS}.
For $p\neq 2,3$ as well, it was proved in~\cite{CMNS} that $L$ is then uniquely determined, and has diamonds in all degrees congruent to $\pm 1$ modulo $6$.

For completeness we briefly discuss the case $k=3$, which is not covered by our arguments above.
If $p>3$ then arguments in~\cite{CMNS} show that $h$ equals $1$ or $2$.
Both possibilities occur.
In fact, the two thin Lie algebras with second diamond $L_3$ constructed in~\cite[Section ~4]{CMNS} have a diamond in each odd degree.
In particular, the next diamond after $L_3$ is $L_5$, whence $h=2$ in those cases.
However, for the metabelian thin Lie algebras constructed in~\cite{GMY} every homogeneous component except for $L_2$ is a diamond,
and hence $h=1$ for those.

\bibliography{References}

\end{document}